\documentclass[11pt]{article}
\usepackage[margin=1in]{geometry}
\usepackage{amsfonts}
\usepackage{hyperref}
\usepackage{amsmath,amsthm}
\usepackage{graphicx}
\usepackage{array}
\usepackage{float}
\usepackage{mathrsfs}
\usepackage{xcolor}
\usepackage{filecontents}
\usepackage{hyperref}
\author{Curtis McDonald}
\title{Technical Note on Relating Scores of Tilted Distributions}
\author{Curtis McDonald}

\theoremstyle{proposition}

\theoremstyle{lemma}

\theoremstyle{corollary}
\newtheorem{corollary}{Corollary}

\theoremstyle{theorem}
\newtheorem{theorem}{Theorem}

\theoremstyle{definition}
\newtheorem{definition}{Definition}

\begin{document}

\maketitle

\begin{abstract}
Recent results in \cite[Lemma 3.3]{moitra2026steering} have shown that for a linear tilt to a reference measure, the scores that would be produced under convolution with a normal variable can be expressed in terms of convolutions of the original density. Here, we extend that result to include constant negative diagonal tilts as well. The relationship follows from relating the denoisers of the two densities, which define the scores via Tweedie formula. A linear tilt results in a location shift to the score operator, while a quadratic tilt results in both a location shift and a time shift. Thus the scores of the tilted density can be understood as the scores of the original convolution process at a different location and noise level. These results are of interest to those in the score based diffusion community, and may lead to better score estimators which take advantage of these tilted score relationships.
\end{abstract}

\section{Introduction}

A core problem that has emerged in modern sampling algorithms has become estimating the score of a convolution of a target density with a normal random variable. Scores produced via convolution with a normal random variable are the natural scores of densities produced under the forward flow Ornstein-Uhlenbeck (OU) process, and estimating these scores is a necessary module to implement the reverse score based diffusion model sampling procedure. We point the reader to seminal work in \cite{song2019generative,song2020score}, a collective overview \cite{yang2023diffusion}, and recent advances in score estimation in \cite{he2024zerothorder},\cite{young2026diffusion} for background on the specific diffusion process. In this work, we will be mostly interested in the problem of score estimation itself, which is a required subroutine to implement these diffusion based sampling methods.

Specifically, define a target density $p(x) \propto e^{-V(x)}$ absolutely continuous with respect to Lebesgue measure in $\mathbb{R}^{d}$ known up to normalization. For any value $\sigma \in [0,1]$, define a new random variable via convolution with an independent normal random variable,
\begin{align}
    U_{\sigma} &= \sqrt{1-\sigma^{2}}X+\sigma Z&X \sim P, Z \sim N(0,I)
\end{align}
$U_{\sigma}$ now has its own density $q(u,\sigma)$ (we may also write $q_{\sigma}(u)$ where appropriate) defined via convolution, and the score of this density is known via Tweedie formula \cite{robbins1956empirical,miyasawa1961empirical,roberts1996exponential,efron2011tweedie} as an expectation over the reverse conditional density.
\begin{align}
    \nabla \log q(u,\sigma) = E_{q}[\frac{u - \sqrt{1-\sigma^{2}}x}{\sigma^{2}}|u,\sigma]
\end{align}
The ability to estimate $\nabla \log q(u,\sigma)$ to sufficient accuracy is the core module that drives score based diffusion models. A very interesting result was recently pointed out in \cite{moitra2026steering} relating the scores of tilted densities. Define $\tilde{p}$ as a linear tilt to the original measure $p$,
\begin{align}
    \tilde{p} &\propto p(x)e^{ v^{T}x}\quad\quad~s.t.~\quad\quad \int p(x)e^{ v^{T}x} dx < \infty
\end{align}
Then define a new variable under convolution with $\tilde{p}$ and a normal,
\begin{align}
    W_{\gamma} &= \sqrt{1-\gamma^{2}}\tilde{X}+\gamma Z&\tilde{X} \sim \tilde{P}, Z \sim N(0,I)
\end{align}
Then $W_{\gamma}$ has its own density $g(w,\gamma)$. If one wants to sample $\tilde{p}$ via a score based diffusion model, one needs an accurate estimator of the scores $\nabla \log g(w,\gamma)$. In \cite[Lemma 3.3]{moitra2026steering}, it is shown the scores of $g(w,\gamma)$ are exactly the scores of $q(u,\sigma)$ under a location shift and affine shift,
\begin{align}
    \nabla \log g(w,\gamma) = \frac{1}{\sqrt{1-\gamma^{2}}}v+\nabla \log q(w + \frac{\gamma^{2}}{\sqrt{1-\gamma^{2}}}v, \gamma)
\end{align}
If we define ``location shift''
\begin{align}
    w' = w+\frac{\gamma^{2}}{\sqrt{1-\gamma^{2}}}v
\end{align}
Then this relationship can be expressed as
\begin{align}
    \nabla \log g(w,\gamma)&= \frac{1}{\gamma^{2}}(w'-w)+\nabla \log q(w', \gamma)
\end{align}
The score of the tilt at a location $w$ can be found by querying the original score operator at a location $w'$.

Here we extend this result to include quadratic terms. Define a constant negative diagonal tilt as for any $s \geq 0$ and any vector $v$,
\begin{align}
p^{*}(x) = \frac{p(x)e^{-\frac{1}{2}s\|x\|^{2}+v^{T}x}}{E_{p}[e^{-\frac{1}{2}s\|x\|^{2}+v^{T}x}]}\label{tilt_eqn}
\end{align} 
Much like with linear tilts, the scores which would be produced under convolution with $p^{*}$ and a normal density are the same scores that would be produced under convolution with $p$, now with both a location shift $w'$ and a time shift $\sigma'$.

\section{Score Relationship}
Consider a reference measure $p(x)$. Define a new variable $U$ for any $\sigma \in [0,1]$
\begin{align}
U = \sqrt{1-\sigma^{2}}X+\sigma Z, Z \sim N(0,I)
\end{align}
or, equivalently, conditional measure
\begin{align}
q(u|x,\sigma) \sim N(\sqrt{1-\sigma^{2}}x, \sigma^{2}I)
\end{align}
Then $X$ and $U$ share the joint density at a fixed $\sigma$,
\begin{align}
q_{\sigma}(x,u) = p(x)q(u|x,\sigma)&\propto p(x)e^{-\frac{1}{2\sigma^{2}}\|u - \sqrt{1-\sigma^{2}}x\|^{2}}
\end{align}
and we have the reverse conditional density which can be written as
\begin{align}
q(x|u,\sigma)&\propto p(x)e^{\frac{\sqrt{1-\sigma^{2}}}{\sigma^{2}}u^{T}x- \frac{1-\sigma^{2}}{2 \sigma^{2}}\|x\|^{2}}\label{rev_con_density}
\end{align}
this is both a linear tilt and constant negative diagonal tilt to $p(x)$.

Define the marginal density of $U$ at a point $\sigma$ as $q(u,\sigma)$. There is a well known expression for the score $\nabla \log q(u, \sigma)$ under Tweedie formula \cite{robbins1956empirical,miyasawa1961empirical,roberts1996exponential,efron2011tweedie}
\begin{theorem}[Tweedie Formula]
\begin{align}
\nabla\log q(u, \sigma)&= E_{q}[- \frac{u - \sqrt{1-\sigma^{2}}x}{\sigma^{2}}|u]=-\frac{1}{\sigma^{2}} u + \frac{\sqrt{1-\sigma^{2}}}{\sigma^{2}} E_{q}[x|u,\sigma]
\end{align}
\end{theorem}
An important object is called the denoiser, 
\begin{definition}
The denoiser is defined as the reverse conditional expression for $x|u$ at a given $\sigma$.
\begin{align}
    F[u,\sigma] = E_{q}[x|u,\sigma]
\end{align}
\end{definition}
Via Tweedie formula, the score is an affine map of the denoiser. Using the form of the reverse conditional density in \eqref{rev_con_density}, the denoiser is the expectation of $p(x)$ under a specific linear and constant negative diagonal tilt.
\begin{align}
F[u,\sigma]&=\int x  \frac{p(x)e^{\frac{\sqrt{1-\sigma^{2}}}{\sigma^{2}}u^{T}x- \frac{1}{2}(\frac{1}{\sigma}^{2}-1)\|x\|^{2}}}{E_{p}[e^{\frac{\sqrt{1-\sigma^{2}}}{\sigma^{2}}u^{T}x- \frac{1}{2}(\frac{1}{\sigma}^{2}-1)\|x\|^{2}}]}dx
\end{align}
Thus it should be no surprise if $p^{*}$ is itself a linear tilt and a constant negative diagonal tilt to $p$ its denoiser is the denoiser of $p$ under a time and location shift.
\begin{theorem}\label{main_denoiser_thm}
Let $p(x)$ have denoiser $F[u,\sigma]$. Let $p^{*}(x)$ be a linear tilt by a vector $v$ to $p$ and a constant negative diagonal tilt with scaling $s \geq 0$,
\begin{align}
p^{*}(x)&= \frac{p(x)e^{v^{T}x - \frac{1}{2}s \|x\|^{2}}}{E_{p}[e^{v^{T}x - \frac{1}{2}s \|x\|^{2}}]} \label{const_neg_daig_tilt}
\end{align}
Assume $X' \sim p^{*}$ and for any $\sigma \in [0,1]$ define random variable
\begin{align}
U' = \sqrt{1-\sigma^{2}}X'+\sigma Z, Z \sim N(0,I)
\end{align}
Denote the denoiser for $p^{*}(x)$ as $G[u,\sigma]$. Then the denoiser for $p^{*}$ is the denoiser for $p$ evaluated at an affine mapping of the inputs. Define the time shift and location shift,
\begin{align}
&u'=\frac{\sigma^{2}}{\sqrt{(1-\sigma^{2}+s \sigma^{2})(1+s \sigma^{2})}} v+ \sqrt{\frac{1-\sigma^{2}}{(1-\sigma^{2}+s \sigma^{2})(1+s \sigma^{2})}} u \label{u_prime}\\ 
&\sigma'=\sqrt{\frac{\sigma^{2}}{1+s \sigma^{2}}} \label{sigma_prime}
\end{align}
Then the denoiser for $p^{*}$ is the denoiser for $p$ evaluated at these new values
\begin{align}
G[u,\sigma]= F[u', \sigma']
\end{align}
\end{theorem}
\begin{proof}
The proof is an algebra procedure in completing the square and collecting like terms. All terms involved are either linear or quadratic and combine under simple rules. See Appendix.
\end{proof}
This theorem is a statement that the denoiser operator is closed under constant negative diagonal tilts. The denoiser of a density of the form  \eqref{const_neg_daig_tilt} is the same as the denoiser of $p$ evaluated at a location shift $u'$ and time shift $\sigma'$.

The scores, being affine maps of the denoiser via Tweedie formula, results in the score of $p^{*}$ under convolution being the score of $p$ under convolution at a different location and time point, plus an affine map.
\begin{theorem}\label{score_thm}
Using the definition of $u', \sigma'$ in \eqref{u_prime} and \eqref{sigma_prime}, the score of the tilt under convolution can be related back to the score of the original density under convolution
\begin{align}
\nabla \log q^{*}(u,  \sigma)=&-\frac{1}{\sigma^{2}}(u- \sqrt{\frac{1-\sigma^{2}}{1-(\sigma')^{2}}}u')+\frac{(\sigma')^{2}\sqrt{1-\sigma^{2}}}{\sigma^{2}\sqrt{1-(\sigma')^{2}}}\nabla \log q(u', \sigma').\label{score_term_main}
\end{align}
\end{theorem}
\begin{proof}
Using Tweedie formula combined with Theorem \ref{main_denoiser_thm}, gives the expression for the scores once the denoisers are related. See Appendix.
\end{proof}
We refer to the value $u'$ as a location shift in the score, as we are evaluating the previous score at a different spatial location. We refer to the $\sigma'$ value as a time shift, as this is the variance that would be produced at a different time point in an OU process.

It may seem that either when $\sigma = 0,1$, certain quantities in \eqref{score_term_main} are dividing by zero. However, by rearranging the expression that is not the case, it is a well defined equation for all $\sigma$ even at 0 and 1.
\begin{corollary}\label{clean_form}
    With linear tilt vector $v$, quadratic tilt $s \geq 0$, and $u',\sigma'$ as before, the score relationship in \eqref{score_term_main} can also be written as
    \begin{align}
        \nabla \log q^{*}(u,\sigma)&=  v \frac{\sqrt{1-\sigma^{2}}}{1-\sigma^{2}+s\sigma^{2}}-u (\frac{s}{1-\sigma^{2}+s \sigma^{2}})+\frac{1}{\sqrt{1+s \sigma^{2}}}\frac{\sqrt{1-\sigma^{2}}}{\sqrt{1+s \sigma^{2}-\sigma^{2}}}\nabla \log q(u',\sigma')
    \end{align}
    which has no singularities when $\sigma = 0$ or $\sigma = 1$.
\end{corollary}
\begin{proof}
    See Appendex
\end{proof}

If $s = 0$ and we have a purely linear tilt, this results in no time shift and only a location shift. This recovers the expression
\begin{align}
&\sigma' = \sigma, u' = \frac{\sigma^{2}}{\sqrt{1-\sigma^{2}}}v+u,&\nabla_{u}\log q^{*}(u, \sigma)=\frac{1}{\sqrt{1-\sigma^{2}}}v+\nabla \log q(u+\frac{\sigma^{2}}{\sqrt{1-\sigma^{2}}}v, \sigma)
\end{align}
which is the recent result in \cite[Lemma 3.3]{moitra2026steering} which first noted the connection between linear tilts and scores.
An equivalent form is
\begin{align}
    \nabla_{u}\log q^{*}(u, \sigma)=\frac{1}{\sigma^{2}}(u'-u)+\nabla \log q(u', \sigma)
\end{align}
which is more clearly stated in the difference between evaluation points $u'$ and $u$. Note we consider $s\geq 0$, and the time shift equation is
\begin{align}
    \sigma' = \sqrt{\frac{\sigma^{2}}{1+s\sigma^{2}}}\leq  \sigma
\end{align}
thus the ``equivalent'' $\sigma'$ point is strictly less than the $\sigma$ point for a negative tilt. The mapping takes $\sigma \in [0,1]$ and returns a $\sigma' \in [0, \sqrt{\frac{1}{1+s}}]$.

\bibliography{cdc_2026}
\bibliographystyle{plain}

\appendix
\section{Appendix: Collected Proofs}
\begin{proof}[Proof of Theorem \ref{main_denoiser_thm}]

The original denoiser has a specific structure as a particular linear tilt and constant negative diagonal tilt
\begin{align*}
 F[u,\sigma]&=\int x  \frac{p(x)e^{\frac{\sqrt{1-\sigma^{2}}}{\sigma^{2}}u^{T}x- \frac{1}{2}(\frac{1}{\sigma}^{2}-1)\|x\|^{2}}}{E_{p}[e^{\frac{\sqrt{1-\sigma^{2}}}{\sigma^{2}}u^{T}x- \frac{1}{2}(\frac{1}{\sigma}^{2}-1)\|x\|^{2}}]}dx
\end{align*}
The scaling of the quadratic term is a one to one mapping of $\sigma^{2}$. Each $\sigma^{2} \in [0,1]$ defines a unique scaling of the quadratic term, and vice verse. If we define $\rho = \frac{1}{\sigma^{2}}-1$ we can express the denoiser equivalently as
\begin{align*}
F[u,\sigma]&= \int x \frac{p(x)e^{\sqrt{\rho(\rho+1)}u^{T}x-\frac{1}{2}\rho\|x\|^{2}}}{E_{p}[e^{\sqrt{\rho(\rho+1)}u^{T}x-\frac{1}{2}\rho\|x\|^{2}}]}dx\quad \rho = \frac{1}{\sigma^{2}}-1
\end{align*}
The denoiser of $p^{*}$ is then a linear and constant diagonal tilt to what is already a linear and diagonal tilt to the base density. Collecting terms we have 
\begin{align*}
G[u, \sigma]&= \int x\frac{p(x)e^{(v+ \frac{\sqrt{1-\sigma^{2}}}{\sigma^{2}}u')^{T}x}e^{-\frac{1}{2}(\frac{1}{\sigma^{2}}+s-1)\|x\|^{2}}}{E_{p}[e^{(v+ \frac{\sqrt{1-\sigma^{2}}}{\sigma^{2}}u')^{T}x}e^{-\frac{1}{2}(\frac{1}{\sigma^{2}}+s-1)\|x\|^{2}}]}dx
\end{align*}
we then read off $\rho' = \frac{1}{\sigma^{2}}+s-1$. This identifies the new $\sigma'$ point as $(\sigma'^{2}) = 1/(\rho'+1) =\frac{\sigma^{2}}{1+s\sigma^{2}}$. Multiply and divide the linear term by $\sqrt{(\rho')(\rho'+1)}$ we arrive at the equivalent $u'$ evaluation point given in the theorem.
\end{proof}
\begin{proof}[Proof of Theorem \ref{score_thm}]
The two scores have relationships with their denoisers
\begin{align*}
\nabla \log q^{*}(u,\sigma)&= - \frac{1}{\sigma^{2}}u+\frac{\sqrt{1-\sigma^{2}}}{\sigma^{2}}G[u,\sigma]\\
\nabla \log q(u',\sigma')&= - \frac{1}{(\sigma')^{2}}u'+\frac{\sqrt{1-(\sigma')^{2}}}{(\sigma')^{2}}F[u',\sigma']
\end{align*}
Using Theorem \ref{main_denoiser_thm}, we know the appropriate choice of $u',\sigma'$ such that $G[u,\sigma] = F[u',\sigma']$. Substitute the $F$ denoiser for the $G$ denoiser in the score of $\nabla \log q^{*}(u,\sigma)$,
\begin{align*}
\nabla \log q^{*}_{\sigma}(u)&= -\frac{1}{\sigma^{2}}u+\frac{\sqrt{1-\sigma^{2}}}{\sigma^{2}}G[u,\sigma]\\
&= -\frac{1}{\sigma^{2}}u+\frac{\sqrt{1-\sigma^{2}}}{\sigma^{2}}F[u',\sigma']\\
&= -\frac{1}{\sigma^{2}}u+\frac{(\sigma')^{2}\sqrt{1-\sigma^{2}}}{\sigma^{2}\sqrt{1-(\sigma')^{2}}}\left(\frac{\sqrt{1-(\sigma')^{2}}}{(\sigma')^{2}}F[u',\sigma'] \right)
\end{align*}
Add and subtract $\frac{1}{(\sigma')^{2}}u'$ to relate the denoiser back to its score and complete the relationship
\begin{align*}
\nabla \log q^{*}_{\sigma}(u)&= -\frac{1}{\sigma^{2}}u+\frac{(\sigma')^{2}\sqrt{1-\sigma^{2}}}{\sigma^{2}\sqrt{1-(\sigma')^{2}}}(\pm\frac{1}{(\sigma')^{2}}u'+\frac{\sqrt{1-(\sigma')^{2}}}{(\sigma')^{2}}F[u',\sigma'])\\
&= -\frac{1}{\sigma^{2}}(u- \sqrt{\frac{1-\sigma^{2}}{1-(\sigma')^{2}}}u')+\frac{(\sigma')^{2}\sqrt{1-\sigma^{2}}}{\sigma^{2}\sqrt{1-(\sigma')^{2}}}\nabla \log q_{\sigma'}(u')
\end{align*}

\end{proof}

\begin{proof}[Proof of Corollary \ref{clean_form}]
Recall the relationship
\begin{align}
    (\sigma')^{2}&= \frac{\sigma^{2}}{1+s\sigma^{2}}
\end{align}
Substitute this in for $\sigma^{2}$ in equation \ref{score_term_main} and it simplifies as
\begin{align}
 \nabla \log q^{*}(u,\sigma)&= -\frac{1}{\sigma^{2}}(u- \sqrt{\frac{1-\sigma^{2}}{1- \frac{\sigma^{2}}{1+s \sigma^{2}}}}u')+\frac{1}{1+s \sigma^{2}}\frac{\sqrt{1-\sigma^{2}}}{\sqrt{1-\frac{\sigma^{2}}{1+s \sigma^{2}}}}\nabla \log q(u',\sigma')\\
 \nabla \log q^{*}(u,\sigma)&= -\frac{1}{\sigma^{2}}(u- \sqrt{\frac{(1-\sigma^{2})(1+s\sigma^{2})}{1+s \sigma^{2}- \sigma^{2}}}u')+\frac{1}{\sqrt{1+s \sigma^{2}}}\frac{\sqrt{1-\sigma^{2}}}{\sqrt{1+s \sigma^{2}-\sigma^{2}}}\nabla \log q(u',\sigma')
\end{align}
Then note the definition of $u'$
\begin{align}
    u'=\frac{\sigma^{2}}{\sqrt{(1-\sigma^{2}+s \sigma^{2})(1+s \sigma^{2})}} v+ \sqrt{\frac{1-\sigma^{2}}{(1-\sigma^{2}+s \sigma^{2})(1+s \sigma^{2})}} u
\end{align}
Substitute this in for $u'$ and we have
\begin{align}
     \nabla \log q^{*}(u,\sigma)&=v \frac{\sqrt{1-\sigma^{2}}}{1-\sigma^{2}+s\sigma^{2}}-\frac{1}{\sigma^{2}}u (\frac{s \sigma^{2}}{1-\sigma^{2}+s \sigma^{2}})+\frac{1}{\sqrt{1+s \sigma^{2}}}\frac{\sqrt{1-\sigma^{2}}}{\sqrt{1+s \sigma^{2}-\sigma^{2}}}\nabla \log q(u',\sigma') \\
 &=  v \frac{\sqrt{1-\sigma^{2}}}{1-\sigma^{2}+s\sigma^{2}}-u (\frac{s}{1-\sigma^{2}+s \sigma^{2}})+\frac{1}{\sqrt{1+s \sigma^{2}}}\frac{\sqrt{1-\sigma^{2}}}{\sqrt{1+s \sigma^{2}-\sigma^{2}}}\nabla \log q(u',\sigma')
\end{align}
\end{proof}

\end{document}